\documentclass[12pt]{amsart}

\usepackage{amssymb}
\usepackage{bm}
\usepackage{graphicx}
\usepackage{psfrag}
\usepackage[usenames,dvipsnames]{xcolor}
\usepackage{enumerate}
\usepackage{multirow}
\usepackage{url}

\usepackage{comment}

\usepackage{algpseudocode}

\usepackage{mathtools}

\usepackage{blkarray}
\usepackage{tikz-cd}

\newcommand{\excise}[1]{}

\newtheorem{thm}{Theorem}[section]
\newtheorem{lemma}[thm]{Lemma}

\newtheorem{cor}[thm]{Corollary}
\newtheorem{prop}[thm]{Proposition}
\newtheorem{conj}[thm]{Conjecture}

\newtheorem{prob}[thm]{Problem}

\theoremstyle{definition}

\newtheorem{example}[thm]{Example}
\newtheorem{remark}[thm]{Remark}
\newtheorem{defn}[thm]{Definition}

\numberwithin{equation}{section}

\renewcommand\>{\rangle}
\newcommand\<{\langle}

\newcommand\NN{\mathbb{N}}

\newcommand\RR{\mathbb{R}}

\newcommand\ZZ{\mathbb{Z}}

\newcommand\kk{\Bbbk}
\newcommand\mm{{\mathfrak m}}

\newcommand\ww{{\mathbf w}}

\newcommand\cc{{\mathbf c}}

\newcommand\free{{\mathcal F}}

\newcommand\ideal[1]{\langle #1 \rangle}

\DeclareMathOperator\Tor{Tor} 

\DeclareMathOperator\rank{rank} 

\DeclareMathOperator\Ap{Ap} 

\newcommand{\grm}{\operatorname{gr_\mm}}

\newcommand{\grmR}{{\grm(R)}}

\newcommand\s{\scriptstyle}
\newcommand\phm{\phantom{-}}
\newcommand\rlm{-}

\newcommand\filleftmap{\mathord\leftarrow \mkern-6mu
	\cleaders\hbox{$\mkern-2mu \mathord- \mkern-2mu$}\hfill
	\mkern-6mu \mathord-}

\begin{document}

\mbox{}
\title[Infinite free resolutions over numerical semigroup algebras]{Infinite free resolutions over numerical \\ semigroup algebras via specialization}

\author[T.~Gomes]{Tara Gomes}
\address{School of Mathematics\\University of Minnesota\\Minneapolis, MN 55455}
\email{gomes072@umn.edu}

\author[C.~O'Neill]{Christopher O'Neill}
\address{Mathematics Department\\San Diego State University\\San Diego, CA 92182}
\email{cdoneill@sdsu.edu}

\author[A.~Sobieska]{Aleksandra Sobieska}
\address{Mathematics Department\\University of Wisconsin Madison\\Madison, WI 53706}
\email{asobieska@wisc.edu}

\author[E.~Torres D\'avila]{Eduardo Torres D\'avila}
\address{School of Mathematics\\University of Minnesota\\Minneapolis, MN 55455}
\email{torre680@umn.edu}

\date{\today}

\begin{abstract}
Each numerical semigroup $S$ with smallest positive element $m$ corresponds to an integer point in a polyhedral cone $C_m$, known as the Kunz cone.  The~faces of $C_m$ form a stratification of numerical semigroups that has been shown to respect a number of algebraic properties of $S$, including the combinatorial structure of the minimal free resolution of the defining toric ideal $I_S$.  
In this work, we prove that the structure of the infinite free resolution of the ground field $\Bbbk$ over the semigroup algebra $\Bbbk[S]$ also respects this stratification, yielding a new combinatorial approach to classifying homological properties like Golodness and rationality of the poincare series in this setting.  Additionally, we give a complete classification of such resolutions in the special case $m = 4$, and demonstrate that the associated graded algebras do not generally respect the same stratification.  
\end{abstract}

\maketitle


\section{Introduction}
\label{sec:intro}

A \emph{numerical semigroup} is a cofinite, additively closed set $S \subseteq \ZZ_{\ge 0}$ containing 0.  We~often specify a numerical semigroup via a list of generators, i.e.,
$$S = \<n_1, \ldots, n_k\> = \{z_1n_1 + \cdots + n_kz_k : z_i \in \ZZ_{\ge 0}\}.$$
Any numerical semigroup has a unique minimal generating set under inclusion; the number $\mathsf e(S)$ of minimal generators is called the \emph{embedding dimension} of $S$, and the smallest generator $\mathsf m(S) = \min(S \setminus \{0\})$ is called the \emph{multiplicity} of $S$.  See~\cite{numerical} for a thorough introduction.  

A number of recent papers examine a family of convex rational polyhedral cones $C_m$, one for each integer $m \ge 2$, for which each numerical semigroup with multiplicity $m$ corresponds to an integer point in $C_m$.  Originally introduced in~\cite{kunz}, much of the recent work centers around the face structure of $C_m$ and the observation that two numerical semigroups $S$ and $T$ correspond to points on the interior of the same face of $C_m$ if and only if certain subsets of their divisibility posets coincide~\cite{wilfmultiplicity,kunzfaces1}.  
Geometric results in this direction include, for instance, a classification of the faces of $C_m$ containing points corresponding to certain well-studied families of numerical semigroups~\cite{kunzfaces2} and a realization of the faces of $C_m$ as cones in the Gr\"obner fans of certain lattice ideals~\cite{kunzfaces4}.  

Of particular relevance to the present manuscript are two recent works concerning the defining toric ideal $I_S$ of a numerical semigroup $S$.  
Kunz observed in~\cite{kunz} that if $S$ and $T$ lie on the interior of the same face of $C_m$, then the Betti numbers of $I_S$ and $I_T$ coincide.  
This result was strengthened in~\cite{kunzfaces3} to show that there exist minimal binomial generating sets for $I_S$ and $I_T$ that coincide except in the exponent of the variable with degree $m$, and further extended in~\cite{kunzfiniteres} to show the existence of minimal free resolutions of $I_S$ and $I_T$ whose matrix entries coincide except for the exponents of that same variable.  Said another way, the minimal resolutions of $I_S$ and $I_T$ are seen to have identical structure when $S$ and $T$ lie in the interior of the same face of $C_m$.  

Broadly, these results identify a novel method of combinatorially classifying the structure of minimal free resolutions of the defining toric ideals of numerical semigroups with a given multiplicity $m$.  This family of toric ideals has already seen some headway over the general setting, due in part to the unique set of tools numerical semigroups have to offer (see the survey~\cite{nsbettisurvey} and the references therein), and stratification by the faces of $C_m$ is an addition to this toolkit.  Indeed, this approach is already being put to use in the classification of possible Betti numbers for fixed $m$~\cite{minprescard}.  

This article, in the spirit of~\cite{hilbertovertoric,binomialfreeresnormal,shellmonoid2,shellmonoid}, seeks to utilize a combinatorial approach in understanding infinite free resolutions over numerical semigroup algebras. Such resolutions over singular rings are typically more complicated than those over polynomial rings, perhaps most evidently because such resolutions are almost always infinite. Some resolutions are known, in particular for the field $\kk$ \cite{anickserrecounterexample,burkehigherhomotopies,frobergpoincare,maclanehomology,priddykoszul,tateres}, either for special types of rings or without guarantee of minimality, but by and large these resolutions remain mysterious. Further details and references can be found in the surveys~\cite{avramovlectures,infinitefreeressurvey}.  

Following a similar roadmap to that of the finite case in~\cite{kunzfiniteres}, we construct in Section~\ref{sec:infiniteaperyres} an infinite free resolution of~$\kk$ over any numerical semigroup ring~$R$, which we call the \emph{infinite Ap\'ery resolution} of $\kk$, that is minimal if and only if $S$ is MED.  The maps in this resolution are seen to allow for simultaneous minimization for all semigroups $S$ and $T$ residing in the interior of the same face of $C_m$, resulting in a uniformity of resolution structure analogous to that obtained for the finite case in~\cite{kunzfiniteres}.  

One of the primary consequences of our results is that the Kunz cone does indeed allow one to classify resolutions of $\kk$ over numerical semigroup rings combinatorially. This combinatorial classification does not \textit{a priori} appeal to usual homological notions relevant in infinite resolutions, like complete intersection, Golod, and Koszul, yet some homological properties whose classifications remain elusive in general, such as Golodness and rationality of the Poincare series, respect this stratification.

As a demonstration of this approach, Section~\ref{sec:m4spec} exhibits a minimal free resolution whenever $S$ has multiplicity 4, obtained by considering each face of $C_4$ in turn.  
After a brief demonstration in Section~\ref{sec:associatedgraded} that the associated graded rings $\grmR$ do not seem to obey the stratification imposed by the Kunz cone, we close by identifying several future research directions in Section~\ref{sec:futurework}.  


\section{Background and Setup}
\label{sec:background}

Throughout this paper, fix a numerical semigroup $S \subseteq \ZZ_{\ge 0}$ with multiplicity 
$$\mathsf m(S) = \min(S \setminus \{0\}) = m,$$
and write
\begin{align*}
\Ap(S)
&= \{n \in S : n - m \notin S\} \\
&= \{0, a_1, \ldots, a_{m-1}\}
\end{align*}
for the set consisting of the minimal element of $S$ from each equivalence class modulo~$m$, with $a_i \equiv i \bmod m$ for each $i$.  In particular, 
$$S = \<m, a_1, \ldots, a_{m-1}\>,$$
though this generating set need not be minimal (e.g., if $a_i + a_j = a_{i+j}$ for some $i, j$).  If this generating set is minimal, we say $S$ has \emph{maximal embedding dimension} (MED).  For convenience, define $a_0 = m$.  

Grade $\kk[y, x_1, \ldots, x_{m-1}]$ by $S$ with $\deg(y) = m$ and $\deg(x_i) = a_i$.  Let $x_0 = y$.  Consider the ring homomorphism 
\begin{align*}
\varphi:\kk[y, x_1, \ldots, x_{m-1}] &\longrightarrow \kk[t] \\
x_i &\longmapsto t^{a_i},
\end{align*}
and let $I_S = \ker(\varphi)$, called the \emph{defining toric ideal} of $S$, and $R = \kk[y, x_1, \ldots, x_{m-1}]/I_S$.  It is known (see, e.g., \cite[Lemma~3.1]{kunzfiniteres}) that
\begin{equation}\label{eq:medbinomials}
I_S = \<x_ix_j - y^{b_{ij}}x_{i+j} : 1 \le i \le j \le m-1\>
\end{equation}
where each $b_{ij} = \tfrac{1}{m}(a_i + a_j - a_{i+j}) \ge 0$ is determined by the grading, and that the above is a minimal generating set if and only if $S$ is MED.  
Note the definition of $b_{ij}$ here differs from that in~\cite{kunzfiniteres}; the discrepency in matrix entries necessitates this distinction.  

\begin{example}\label{e:definingideal}
If $S = \<4, 5, 7\>$, then $\Ap(S) = \{0, 5, 10, 7\}$, and 
\begin{align*}
I_S
&= \<x_1^2 - x_2, x_2^2 - y^5, x_3^2 - yx_2, x_1x_2 - y^2x_3, x_1x_3 - y^3, x_2x_3 - y^3x_1\>
\\
&= \<x_1^2 - x_2, x_1^3 - y^2x_3, x_3^2 - yx_1^2, x_1x_3 - y^3\>
\end{align*}
We emaphasize that $R$ is the quotient by $I_S$, so
$$R = \kk[y, x_1, x_2, x_3]/I_S \cong \kk[y, x_1, x_3]/\<x_1^3 - y^2x_3, x_3^2 - yx_1^2, x_1x_3 - y^3\>,$$
and in particular $x_1^2 - x_2 = 0$ and $x_1x_2 - yx_3 = x_1^3 - x_3$ in $R$.  
\end{example}

A given collection of integers $a_1, \ldots, a_{m-1} \ge m$ with $a_i \equiv i \bmod m$ for each $i$ comprise the nonzero elements of the Ap\'ery set of a numerical semigroup $S$ if and only if the point $(a_1, \ldots, a_{m-1})$ lies in the cone $\mathcal C_m \subseteq \RR_{\ge 0}^{m-1}$ defined by the inequalities
$$
x_i + x_j \ge x_{i+j}
\qquad \text{for each} \qquad
i, j = 1, \ldots, m-1
\qquad \text{with} \qquad
i + j \ne 0,
$$
called the \emph{Kunz cone} due to~\cite{kunz}.  
Each facet equality $a_i + a_j = a_{i+j}$ occurs precisely when $b_{ij} = 0$, and this is not possible in the case $i + j = 0$.  As such, the \emph{Kunz poset} $P = (\ZZ_m, \preceq)$ of $S$, which sets $i \preceq j$ for distinct $i, j \ne 0$ if $a_j - a_i \in \Ap(S)$, is determined by which facets of $C_m$ contain $(a_1, \ldots, a_{m-1})$.   

With this in mind, if $(a_1, \ldots, a_{m-1})$ lies in the relative interior $F^\circ$ of a face $F \subseteq C_m$, we conflate notation and say $S$ lies in $F$.  In particular, a numerical semigroup $T$ lies in the same face of $C_m$ as $S$ if and only if its Kunz poset is identical to that of $S$.  

\begin{example}\label{e:kunzposet}
The Kunz posets of each face of $\mathcal C_3$ and $\mathcal C_4$ are depicted in Figure~\ref{fig:m3m4cones}.  The latter has defining inequalities
$$
2a_1 \ge a_2, 
\qquad
a_1 + a_2 \ge a_3, 
\qquad
a_2 + a_3 \ge a_1,
\qquad \text{and} \qquad
2a_3 \ge a_2.
$$
Resuming notation from Example~\ref{e:definingideal}, the point $(5,10,7) \in \mathcal C_4$ lies in the interior of the upper-left facet, and the Kunz poset depicted therein reflects that $x_1^2 - x_2 \in I_S$.  

The numerical semigroup $T = \<4, 13, 31\>$ lies in the same facet of $C_4$ as~$S$ since $\Ap(T) = \{0, 13, 26, 31\}$ with $26 = 2 \cdot 13$ and $13 + 26 > 31$.  In the generating sets
\begin{align*}
I_S &= \<x_1^2 - x_2, x_1^3 - y^2x_3, x_3^2 - yx_1^2, x_1x_3 - y^3\>, \\
I_T &= \<x_1^2 - x_2, x_1^3 - y^3x_3, x_3^2 - y^7x_1^2, x_1x_3 - y^{10}\>,
\end{align*}
each exponent of $y$ equals $b_{ij}$ for some $i, j$.  More specifically, both generating sets consist of the binomials
$$x_1^2 - x_2, 
\qquad
x_1^3 - y^{b_{12}}x_3, 
\qquad
x_3^2 - y^{b_{33}}x_1^2, 
\qquad \text{and} \qquad
x_1x_3 - y^{b_{13}+1},$$
and only the values $b_{12}$, $b_{13}$, and $b_{33}$ depend on which semigroup is chosen.  
\end{example}

\begin{figure}[t]
\includegraphics[width=2.5in]{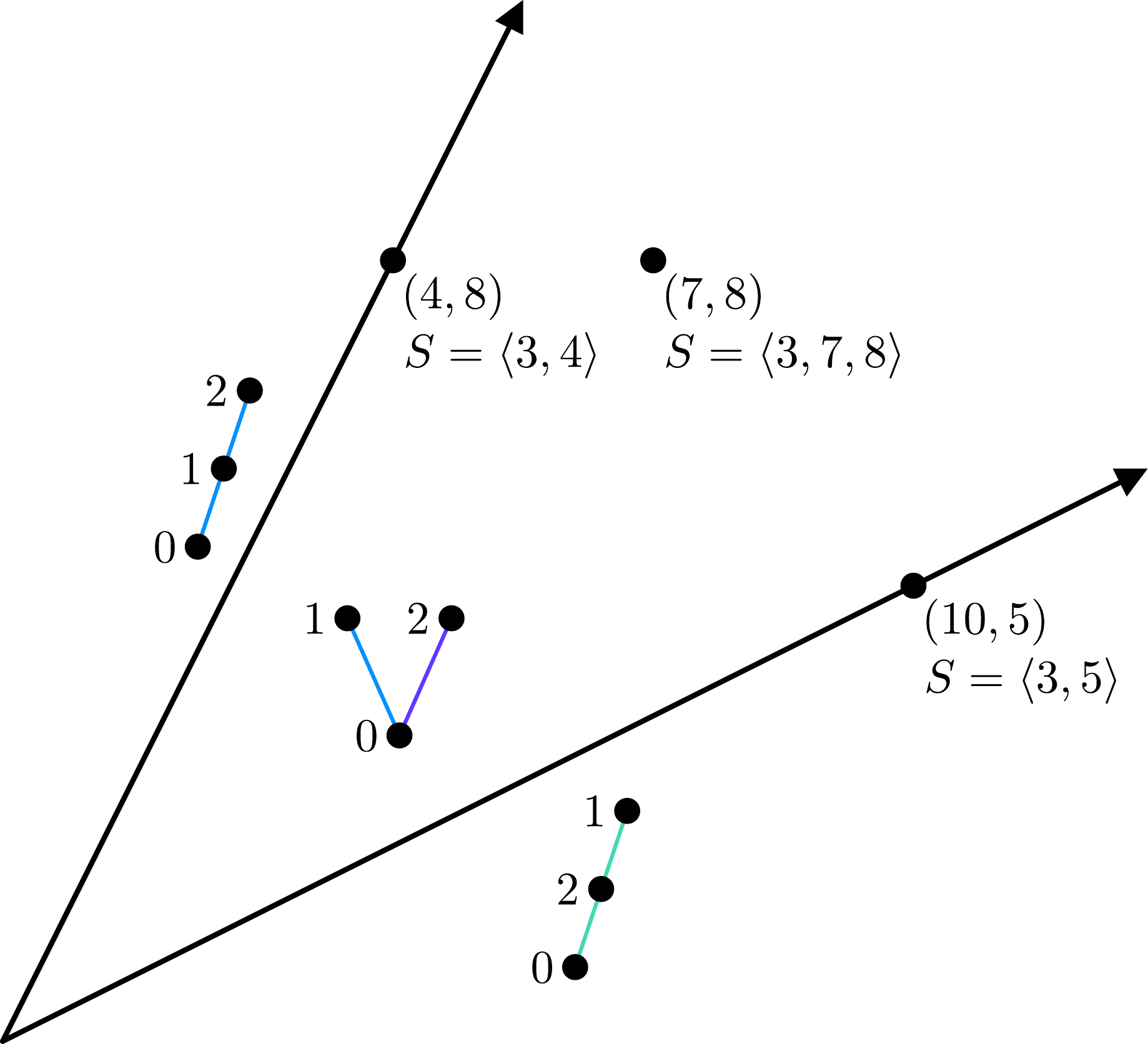}
\qquad\qquad
\includegraphics[width=2.5in]{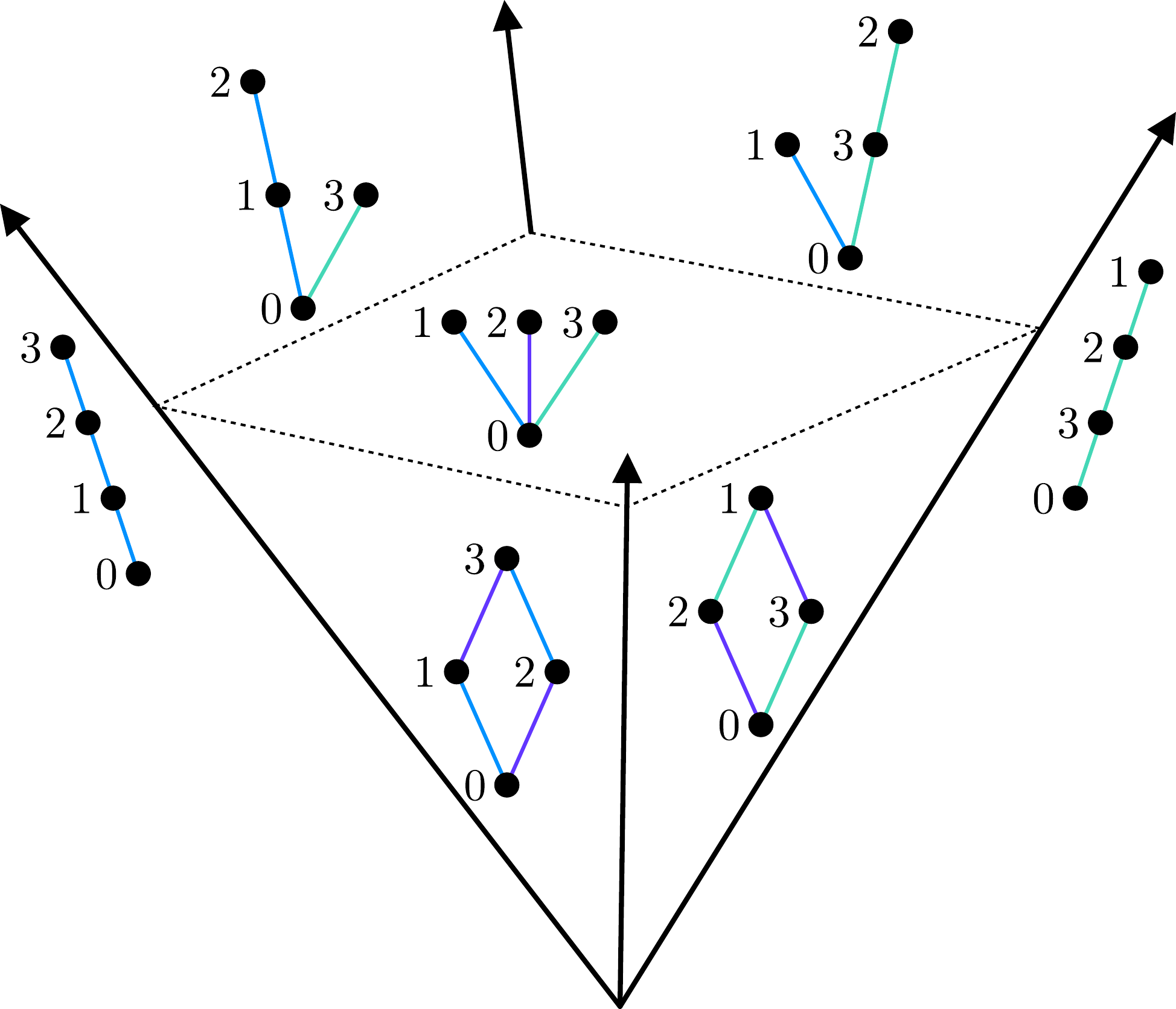}
\caption{The Kunz cones $C_3$ (left) and $C_4$ (right).  Each face containing numerical semigroups is labeled with its Kunz poset.}
\label{fig:m3m4cones}
\end{figure}

The second paragraph of Example~\ref{e:kunzposet} illustrates \cite[Corollary~5.16]{kunzfaces3}:\ if $S$ and $T$ are numerical semigroups in the same face of $C_m$, then there exist minimal binomial generating sets for $I_S$ and $I_T$ that coincide except in the exponents of $y$.  Each exponent of $y$ therein equals $b_{ij}$ for some $i,j$, and an expression for the generating set that is consistent across the face of $C_m$ containing $S$ and $T$ (such as the one in the last centered line of Example~\ref{e:kunzposet}) can be recovered directly from the Kunz poset of $S$ and $T$.  


This result was further extended in~\cite[Theorem~4.4]{kunzfiniteres}:\ there exist minimal free resolutions of $I_S$ and $I_T$ whose matrix entries coincide except in the exponents of~$y$.  Just~like the generating sets for $I_S$ and $I_T$ constructed in~\cite{kunzfaces3}, each exponent of $y$ in the matrices in these resolutions coincides with~$b_{ij}$ for some $i, j$.  This~result is best illustrated with an example; the following is an exerpt from \cite[Example~4.2]{kunzfiniteres}, though we encourage the reader to consult the full resolution depicted therein (the reader may also wish to consult~\cite{cca} for background on graded free resolutions, and~\cite{nsbettisurvey} in the context of numerical semigroups).  

\begin{example}\label{e:kunzfiniteres}
Resuming notation from Example~\ref{e:kunzposet}, there exist minimal free resolutions for $I_S$ and $I_T$ such that within each, the second matrix equals
$$
\left[
\begin{array}{@{}lllll@{}}
\phm x_1^3 - x_3 y^{b_{12}} & \phm x_3^2 - x_1^2 y^{b_{33}} & \phm x_1x_3 - y^{b_{13}+1} & \phm                        & \phm                      \\
\rlm (x_1^2 - x_2)            & \phm                        & \phm                     & \phm y^{b_{33}}             & \rlm x_3                  \\
\phm                          & \rlm (x_1^2 - x_2)          & \phm                     & \phm x_1                    & \rlm y^{b_{12}}           \\
\phm                          & \phm                        & \rlm (x_1^2 - x_2)       & \rlm x_3                    & \phm x_1^2                \\
\end{array}
\right]
$$
and only the exponents $b_{12}$, $b_{13}$, and $b_{33}$ depend on which semigroup is chosen.  
\end{example}

Unlike the construction of binomial generating sets, which can be done explicitly from the Kunz poset by~\cite{kunzfaces3}, the proof of~\cite[Theorem~4.4]{kunzfiniteres} is non-constructive.  Instead, a~explicit free resolution is constructed for $I_S$, called the \emph{Ap\'ery resolution}, that is minimal if and only if $S$ is MED.  It is then proven that there exists a sequence of row and column operations, dependent only on the face of $C_m$ containing $S$, that may be performed to obtain the minimal free resolution of $I_S$ from the Ap\'ery resolution.

\section{The infinite Ap\'ery resolution}
\label{sec:infiniteaperyres}

In this section, we construct a free resolution for $\kk$ over $R$ that is minimal if and only if $S$ is MED. 

\begin{defn}\label{d:infiniteaperyres}
The \textit{infinite Ap\'ery resolution} of $R$ is the free resolution
\begin{equation}
\mathcal F_\bullet: 0 \longleftarrow F_0 \longleftarrow F_1 \longleftarrow F_2 \longleftarrow \cdots,
\end{equation}
where each $F_d$ is generated as a module as 
$$
F_d = \<e_\ww : \ww = (w_1, w_2, \ldots, w_d) \in \ZZ_m^{d}\>,
\qquad \text{where} \qquad
\deg(e_\ww)
= \sum_{i = 1}^d \deg(x_{w_i})
$$
and $e_\ww = 0$ whenever $w_i = 0$ for some $i > 1$.  
One can readily check that $F_0 = R$ and 
$$\rank F_d = m(m-1)^{d-1}$$
for each $d \ge 1$.  Each map $\partial: F_d \to F_{d-1}$ in $\mathcal F_\bullet$ is given by
$$e_{\ww} \mapsto 
x_{w_d} e_{\widehat \ww} + \sum_{i=1}^{d-1} (-1)^{d-i} y^{b_{w_iw_{i+1}}} e_{\tau_i \ww}$$
for each $d \ge 1$, where ${\widehat \ww}= (w_1, w_2, \ldots, w_{d-1})$ and $\tau_i:\ZZ_m^{d} \to \ZZ_m^{d-1}$ is given by 
$$\tau_i \ww = (w_1, \ldots, w_{i-1}, w_i + w_{i+1}, w_{i+2}, \ldots, w_d).$$
Note that if $w_i = 0$ for some $i > 1$, then $\tau_i \ww = \tau_{i-1} \ww$, and $e_{\tau_j \ww} = 0$ whenever $j \notin \{i, i-1\}$, so under the above definition $\partial_d e_\ww = 0$ whenever $e_\ww = 0$.  Figure~\ref{fig:m4res} depicts the beginning of the resolution for $m = 4$.  

Lastly, we say a term $ct^n e_{\ww} \in F_d$ has \emph{class} $k$ if $n \equiv k \bmod m$.  In particular, any class~0 term can be written as $y^q e_{\ww}$ for some $q \ge 0$.  
\end{defn}

\begin{figure}[t]
\includegraphics[width=6.0in]{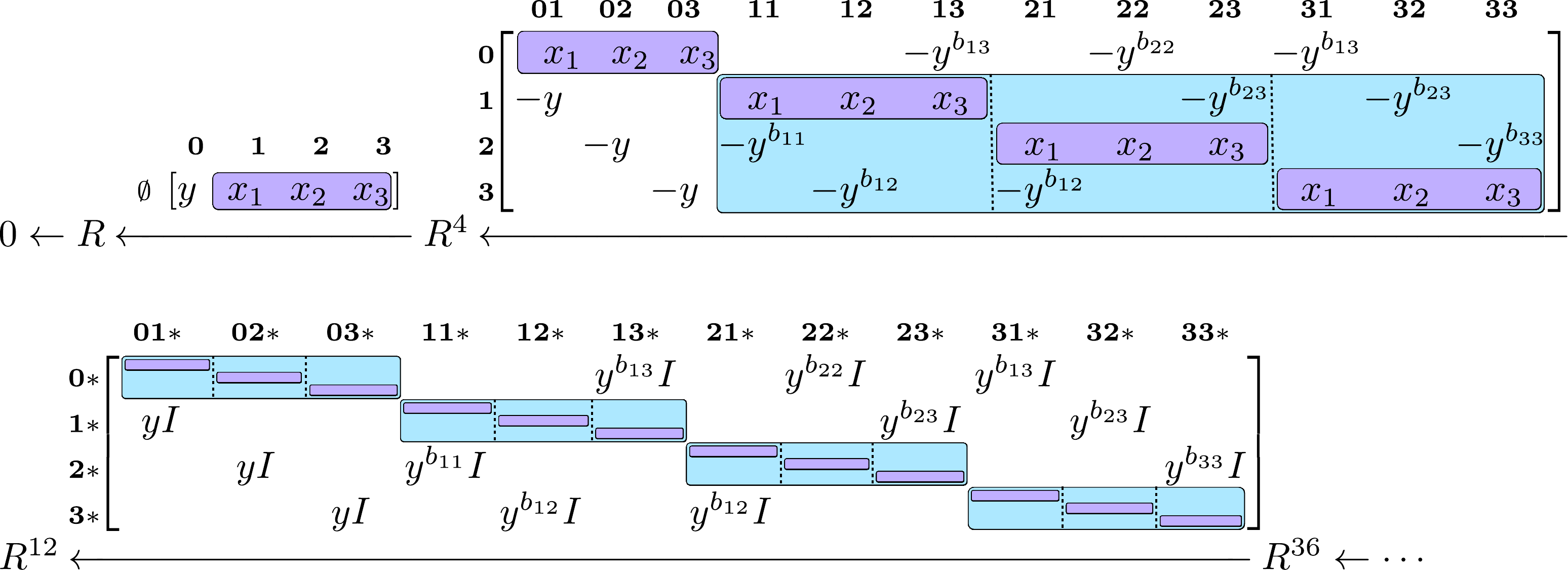}
\caption{The infinite multiplicity resolution for $m = 4$.}
\label{fig:m4res}
\end{figure}

\begin{thm}\label{t:medresolution}
The complex $\mathcal F_\bullet$ is a resolution.  Additionally, $\mathcal F_\bullet$ is minimal if and only if $S$ is MED.  
\end{thm}

\begin{proof}
We first check that $\mathcal F_\bullet$ is a complex.  If $d = 2$, then
\begin{align*}
\partial^2 e_{uv}
&= x_v \partial e_u - y^{b_{uv}} \partial e_{u+v}
= x_v x_u - y^{b_{uv}} x_{u+v} = 0.  
\end{align*}
On the other hand, fix $d \ge 2$ and $\ww \in \ZZ_m^d$ with $w_i \ne 0$ whenever $i > 1$.  Since $\partial e_\ww$ is homogeneous under the fine grading by $S$, all occurances of a given basis vector in $\partial^2 e_\ww$ have identical monomial coefficient, so it suffices to check the $\kk$-coefficients of each sum to~0.  To this end, since
\begin{enumerate}[(i)]
\item 
$\widehat{\tau_{d-1} \ww} = \widehat{\widehat \ww}$, 

\item 
$\tau_i \widehat \ww = \widehat{\tau_i \ww}$ for $1 \le i \le d-2$, and

\item 
$\tau_i \tau_j \ww = \tau_{j-1} \tau_i \ww$ for $1 \le i < j \le d-1$, 

\end{enumerate}
each basis vector arising in $\partial^2 e_\ww$ appears exactly twice, and with opposite signs.  

We next prove $\mathcal F_\bullet$ is a resolution.  The claim is clear for $d = 1$, so fix $d \ge 2$ and $f \in \ker \partial_d$.  We claim that if every term in $f$ has class 0, then $f = 0$.  Indeed, for any nonzero term $y^q e_{\ww}$ in $f$, the image $\partial y^q e_{\ww}$ has exactly one term with nonzero class, namely $y^q x_{w_d} e_{\widehat \ww}$.  Moreover, for any two nonzero terms $y^q e_{\ww}$ and $y^{q'} e_{\ww'}$ in $f$, if~$y^q x_{w_d} e_{\widehat \ww} = y^{q'} x_{w_d'} e_{\widehat \ww'}$, then $\widehat \ww = \widehat \ww'$, $w_d = w_d'$, and $q = q'$.  
As such, the image of any nonzero term in $f$ has a term that cannot be cancelled by the image of any other nonzero term in $f$.  This necessitates $f = 0$, proving the claim.  

Now, suppose $f$ has a nonzero term $t^n e_\ww$ with nonzero class.  Write $n = qm + a_r$ with $q, r \in \ZZ$ and $0 < r < m$, so that $t^n = y^q x_r$.  The image $\partial y^q e_{\ww r}$ has exactly one nonzero term with nonzero class, namely $y^q x_r e_\ww = t^n e_\ww$, so
$$f - \partial_{d+1} y^q e_{\ww r} \in  \ker \partial_d$$
has one fewer term with nonzero class than $f$ does.  The exactness of $\mathcal F_\bullet$ now follows from induction on the number of terms in $f$ with nonzero class.  

For the final claim, recall that a graded resolution is minimal if and only if the matrices for $\partial$ contain no nonzero constant entries.  The only matrix entries that depend on the particular values of $a_1, \ldots, a_{m-1}$ are the powers of $y$, and their exponents $b_{ij}$ are all positive precisely when $S$ is MED.  This completes the proof.  
\end{proof} 

The following consequence of Theorem~\ref{t:medresolution} has identical proof to \cite[Theorem~4.4]{kunzfiniteres}.  

\begin{cor}\label{c:specialization}
Consider the set 
$$
  \mathcal M
  =
  \{x_i : 1 \le i \le m-1\} \cup \{y^{b_{ij}} : 1 \le i, j \le m - 1\}
$$
of formal symbols appearing as matrix entries in infinite Ap\'ery resolutions.  Fix a face $F$ of $C_m$.  There is a sequence of matrices, whose entries are $\kk$-linear combinations of formal products of elements of $\mathcal M$, with the following property:\ for each numerical semigroup~$S$ in the relative interior of~$F$, substituting $R$-variables and the values $b_{ij}$ for $S$ into the entries of each matrix yields boundary maps for a graded minimal free resolution of $\kk$ over $R$.
\end{cor}

The following results focus on the Betti numbers of these and related resolutions, which are encoded in the \textit{Poincar\'e series} of $R$, i.e., the formal power series
\[P^R_\kk(z) \coloneqq \sum\limits_{i=0}^\infty \beta_i^R(\kk) z^i.\]

\begin{remark}\label{r:kunzsecret}
Proposition~\ref{p:uniformpoincare} and Corollary~\ref{c:golod} below can be seen as consequences of \cite[Example~2.4]{kunz} using standard tools~\cite[Proposition~3.3.5]{avramovlectures}, and some interest was shown therein for classification of the Betti numbers and syzygies of $I_S$ for numerical semigroups $S$ with fixed multiplicity~$m$.  The~infinite Ap\'ery resolution is a step in this direction for the Betti numbers $\beta_i^R(\kk)$; we include short proofs of these results here as a demonstration of this fact.  
\end{remark}

\begin{prop}\label{p:uniformpoincare}
If $R$ and $R'$ are the respective semigroup rings of numerical semigroups $S$ and $S'$ in the same face of the Kunz cone $C_m$, then they have the same Poincar\'e series, that is, $P^R_\kk(z) = P^{R'}_\kk(z)$. In particular, the rationality of the Poincar\'e series of $R$ is determined by the face of $C_m$ containing $S$.  
\end{prop}

\begin{proof}
The following proof is much in the spirit of the proof of \cite[Theorem~4.3]{kunzfiniteres}. 

Let $\free_\bullet$ and $\free_\bullet'$ be the infinite Ap\'ery resolutions for $\kk$ over $R$ and $R'$ respectively. If $\free_\bullet$ and $\free_\bullet'$ are not minimal, then they have $\pm 1$s in identical places in the differentials, since the only way for an entry to be $1$ in such a resolution is if some $b_{ij}=0$, which happens simultaneously for $S$ and $S'$ in the same face. Betti numbers can be realized as the dimension of $\Tor$, so we consider the homology of $\free_\bullet \otimes \kk$ and $\free_\bullet' \otimes \kk$. These tensored complexes are identical -- any element of the maximal ideal vanishes, leaving matching $\pm 1$s in matching places, which means the images and kernels of these new differentials are also identical. Now we can see that 
\begin{align*}
\beta_i^R(\kk)
&= \dim_\kk \Tor_i^R(\kk, \kk)
= \dim_\kk H_i(\free_\bullet \otimes \kk)
= \dim_\kk H_i(\free_\bullet' \otimes \kk)
\\
&= \dim_\kk \Tor_i^{R'}(\kk, \kk)
= \beta_i^{R'}(\kk),
\end{align*}
as desired.  
\end{proof}

Use $Q$ to denote the polynomial ring $\kk[y,x_1, \ldots, x_{m-1}]$ graded as in the definition of $R = Q/I_S$.  Recall that $R$ is \textit{Golod} when the Poincar\'e series $P^R_\kk(z)$ satisfies 
\begin{equation}\label{eq:golodequality}
P^R_\kk(z) = \frac{(1+z)^n}{1 - z^2 P^Q_I(z)},
\end{equation}
where $P^Q_I(z) = \sum\limits_{i \geq 0} \beta_i^Q(I) z^i$ and $n$ is the number of minimal generators of $S$. 

\begin{cor}\label{c:golod}
Golodness is a uniform property across the faces of $C_m$, that is, either all semigroups from a given face of $C_m$ give rise to Golod semigroup rings or none do. 
\end{cor}

\begin{proof}
Let $S$ and $S'$ be two semigroups from the same face of $C_m$ and let $R$ and $R'$ be their numerical semgroup rings, respectively. By Proposition~\ref{p:uniformpoincare}, $P_\kk^R(z) = P_\kk^{R'}(z)$, and it is known from \cite[Corollary~4.5]{kunzfiniteres} that $\beta_d(I_S) = \beta_d(I_{S'})$ for all $d$, so $P^Q_{I_S}(z) = P^Q_{I_{S'}}(z)$.  Therefore, the left and right hand sides of \eqref{eq:golodequality} are identical for $S$ and $S'$.
\end{proof}

\begin{cor}\label{c:medgolod}
If $S$ is MED, then $R$ is Golod.  
\end{cor}

\begin{proof}
The resolution $\free_\bullet$ is minimal when $S$ is MED, so
\[
P^R_\kk(z)
= 1 + \sum\limits_{i=1}^\infty (m)(m-1)^{i-1} z^i
= 1 + \frac{mz}{1-(m-1)z}
= \frac{1+z}{1-(m-1)z}.
\]
On the other hand, the resolution of $I$ over $S$ is given by the Ap\'ery resolution~\cite{kunzfiniteres}, so one can readily verify that
\[
\begin{aligned}
P_I^Q(z) &= \sum\limits_{i=0}^{m-2} (i+1) \binom{m}{i+2}z^i
= \frac{1 - (1 + z)^{m-1}(1 - (m-1)z)}{z^2}
\end{aligned}
\]
and thus $R$ is Golod.  
\end{proof}

\begin{remark}\label{r:notgolod}
It is worth noting that, despite $R$ being Golod in the MED case, the infinite Ap\'ery resolution does \textit{not} use the same basis as the resolution of Golod--Eagon--Shamash~\cite{gulliksenlevin}. This is easily seen in even the second differential, where each column has exactly one $x_i$ entry, whereas any Eagon resolution will have columns with two $x$ entries, since it is based in part on the Koszul complex resolving $\kk$ over $Q$. 
	
\end{remark}

\section{Specialized resolutions for $m = 4$}
\label{sec:m4spec}

Throughout this section, assume $m = 4$, and as before, let $\Ap(S) = \{0, a_1, a_2, a_3\}$ with $a_i \equiv i \bmod 4$.  In what follows, we obtain the minimal resolution of $\kk$ over $\kk[S]$ by considering which face of $\mathcal C_4 \subseteq \RR^3$ contains $S$.  
Up to symmetry under $\ZZ_4^*$ action, $\mathcal C_4$ has two facets and one ray that contain numerical semigroups.  The ray is easy:\ any numerical semigroup in this face has the form $S = \<4, a_1\>$ with $a_1 \equiv 1 \bmod 4$, so $I_S = \<x_1^4 - y^{b_{13}}\>$ is principle and the resolution of $\kk$ over $\kk[S]$ is the one in Figure~\ref{fig:m4rayres} (due~to the periodicity of the matrices therein, it suffices to check exactness in two places to verify this is a resolution).  

When $S$ lies in a facet of $\mathcal C_4$, $S$ has 3 minimal generators, so the resolution of $\kk$ over $\kk[S]$ is more complicated.  The facets in one orbit of the $\ZZ_4^*$-action contain only complete intersection numerical semigroups, while semigroups in the facets in the other orbit are not.  We handle these two cases in turn.  

\begin{figure}[t]
$
\begin{array}{c@{\:}c@{\:}c@{\:}c@{\:}c@{\:}c@{\:}c@{\:}c@{\:}c}
&
\begin{blockarray}{rcc}
	\\
	&
	\s \mathbf{0} &
	\s \mathbf{1} \\
	\begin{block}{r@{\,\,}[*{2}{@{\,\,}l}]}
		\s \mathbf{\emptyset} & 
		y &
		x_1 \\
	\end{block}
\end{blockarray}
&&
\begin{blockarray}{@{}r@{\,\,\,}*{2}{@{}c}}
	&
	\s \mathbf{\ \ 01\ \ } &
	\s \mathbf{\ \ 13\ \ } \\
	\begin{block}{@{}l@{\,\,\,}[*{2}{@{}l}]}
		\s \mathbf{0} & \phm x_1        & \rlm y^{b_{13}}  \\
		\s \mathbf{1} & \rlm y          & \phm x_1^3       \\
	\end{block}
\end{blockarray}
&&
\begin{blockarray}{@{}r@{\,\,\,}*{2}{@{}c}}
	&
	\s \mathbf{\ 013\ } &
	\s \mathbf{\ 131\ } \\
	\begin{block}{@{}l@{\,\,\,}[*{2}{@{}l}]}
		\s \mathbf{01} & x_1^3      & \phm y^{b_{13}}  \\
		\s \mathbf{13} & y          & \phm x_1         \\
	\end{block}
\end{blockarray}
&&
\begin{blockarray}{@{}r@{\,\,\,}*{2}{@{}c}}
	&
	\s \mathbf{\ 0131\ } &
	\s \mathbf{\ 1313\ } \\
	\begin{block}{@{}l@{\,\,\,}[*{2}{@{}l}]}
		\s \mathbf{013} & \phm x_1        & \rlm y^{b_{13}}  \\
		\s \mathbf{131} & \rlm y          & \phm x_1^3       \\
	\end{block}
\end{blockarray}
&
\\[-1em]
0 \leftarrow R & \filleftmap & R^2 & \filleftmap & R^2 & \filleftmap & R^2 & \filleftmap & R^2 \leftarrow \cdots
\end{array}
$
\caption{Specialized resolution when $m = 4$, $a_2 = 2a_1$, and $a_3 = 3a_1$.}
\label{fig:m4rayres}
\end{figure}

\subsection{A facet of $\mathcal C_4$ containing CI numerical semigroups}

For this subsection, suppose $a_3 = a_1 + a_2$ and thus $x_3 = x_1x_2$, so that $S = \<4, a_1, a_2\>$ lies in the lower-left facet of $C_4$ in Figure~\ref{fig:m3m4cones}.  

It turns out the infinite Ap\'ery resolution over $R$ can in this case be minimized leaving only the summands $e_\ww$ for which $w_i \le w_{i+1}$ for each $i$.  This can be seen in the row and column labels on the matrices in Figure~\ref{fig:m4ciface}.  However, note that some columns have more than one $x_i$ entry.  As such, it will be convenient to index the basis vectors in this resolution by multisets $\cc$ of elements of $\{0,1,2\}$.  

To this end, let $\mathcal A$ denote the collection of all finite multisets comprised of elements from $\{0,1,2\}$.  For $\cc \in \mathcal A$, let $c_i$ denote the number of copies of $i$ in $\cc$, and write $|\cc| = c_0 + c_1 + c_2$ for the cardinality of $\cc$.  For each $i = 0, 1, 2$, denote by $\cc + i$ the multiset with one additional copy of $i$, and if $c_i > 0$, denote by $\cc - i$ the multiset with one less copy of $i$.  Define
$$
F_d' = \<e_\cc' : \cc \in \mathcal A, \, |\cc| = d\>,
\qquad \text{where} \qquad
\deg(e_\cc') = \sum_{i = 0}^2 c_i \deg(x_i).
$$
Let $e_\cc' = 0$ if $c_0 > 1$, and for convenience, write $e_{\cc - i}' = 0$ if $c_i = 0$.  Clearly
\[
\rank F_d' = 2d + 1.
\]
Define $\partial:F_d' \to F_{d-1}'$ by
$$
e_\cc' \mapsto \begin{cases}
x_2 e_{\cc-2}' + x_1 e_{\cc-1}' - y^{b_{11}}e_{\cc-11+2}' + (-1)^{d-1}(ye_{\cc-0}' + y^{b_{22}}e_{\cc-22+0}')
& \text{if $2 \mid c_2$;} \\
x_2 e_{\cc-2}' - x_1 e_{\cc-1}' + (-1)^{d-1}(ye_{\cc-0}' + y^{b_{22}}e_{\cc-22+0}')
& \text{if $2 \nmid c_2$,}
\end{cases}
$$
a few matrices of which are depicted in Figure~\ref{fig:m4ciface} (in the row and column labels therein are multisets, e.g., $\mathbf{0111122}$ denotes the multiset $\cc$ with $c_0 = 1$, $c_1 = 4$, and $c_2 = 2$).  

\begin{figure}[t]
$
\begin{array}{c@{\:}c@{\:}c@{\:}c@{\:}c@{\:}c@{\:}c}
&
\begin{blockarray}{rccc}
	\\ \\ \\ \\
	&
	\s \mathbf{0} &
	\s \mathbf{1} &
	\s \mathbf{2} \\
	\begin{block}{r@{\,\,}[*{3}{@{\,\,}l}]}
		\s \mathbf{\emptyset} & 
		y &
		x_1 &
		x_2 \\
	\end{block}
\end{blockarray}
&&
\begin{blockarray}{@{}r@{\,\,\,}*{5}{@{}c}}
	\\ \\
	&
	\s \mathbf{01} &
	\s \mathbf{02} &
	\s \mathbf{11} &
	\s \mathbf{12} &
	\s \mathbf{22} \\
	\begin{block}{@{}l@{\,\,\,}[*{5}{@{}l}]}
		\s \mathbf{0} & \phm x_1        & \phm x_2        & \phm            & \phm            & \rlm y^{b_{22}}  \\
		\s \mathbf{1} & \rlm y          & \phm            & \phm x_1        & \phm x_2        & \phm             \\
		\s \mathbf{2} & \phm            & \rlm y          & \rlm y^{b_{11}} & \rlm x_1        & \phm x_2         \\
	\end{block}
\end{blockarray}
&&
\begin{blockarray}{@{}r@{\,\,\,}*{7}{@{}c}}
	&
	\s \mathbf{011} &
	\s \mathbf{012} &
	\s \mathbf{022} &
	\s \mathbf{111} &
	\s \mathbf{112} &
	\s \mathbf{122} &
	\s \mathbf{222} \\
	\begin{block}{@{}l@{\,\,\,}[*{7}{@{}l}]}
		\s \mathbf{01} & \phm x_1        & \phm x_2        & \phm            & \phm            & \phm            & \phm y^{b_{22}} & \phm            \\
		\s \mathbf{02} & \rlm y^{b_{11}} & \rlm x_1        & \phm x_2        & \phm            & \phm            & \phm            & \phm y^{b_{22}} \\
		\s \mathbf{11} & \phm y          & \phm            & \phm            & \phm x_1        & \phm x_2        & \phm            & \phm            \\
		\s \mathbf{12} & \phm            & \phm y          & \phm            & \rlm y^{b_{11}} & \rlm x_1        & \phm x_2        & \phm            \\
		\s \mathbf{22} & \phm            & \phm            & \phm y          & \phm            & \phm            & \phm x_1        & \phm x_2        \\
	\end{block}
\end{blockarray}
&
\\[-1em]
0 \leftarrow R & \filleftmap & R^3 & \filleftmap & R^5 & \filleftmap
\end{array}
$

$
\begin{array}{c@{\:}c@{\:}c@{\:}c@{\:}c@{}c}
&
\begin{blockarray}{@{}r@{\,\,\,}*{9}{@{}c}}
	&
	\s \mathbf{\ 0111\ } &
	\s \mathbf{\ 0112\ } &
	\s \mathbf{\ 0122\ } &
	\s \mathbf{\ 0222\ } &
	\s \mathbf{\ 1111\ } &
	\s \mathbf{\ 1112\ } &
	\s \mathbf{\ 1122\ } &
	\s \mathbf{\ 1222\ } &
	\s \mathbf{\ 2222\ } \\
	\begin{block}{@{}l@{\,\,\,}[*{9}{@{}l}]}
		\s \mathbf{011} & \phm x_1        & \phm x_2        & \phm            & \phm            & \phm            & \phm            & \rlm y^{b_{22}} & \phm            & \phm            \\
		\s \mathbf{012} & \rlm y^{b_{11}} & \rlm x_1        & \phm x_2        & \phm            & \phm            & \phm            & \phm            & \rlm y^{b_{22}} & \phm            \\
		\s \mathbf{022} & \phm            & \phm            & \phm x_1        & \phm x_2        & \phm            & \phm            & \phm            & \phm            & \rlm y^{b_{22}} \\
		\s \mathbf{111} & \rlm y          & \phm            & \phm            & \phm            & \phm x_1        & \phm x_2        & \phm            & \phm            & \phm            \\
		\s \mathbf{112} & \phm            & \rlm y          & \phm            & \phm            & \rlm y^{b_{11}} & \rlm x_1        & \phm x_2        & \phm            & \phm            \\
		\s \mathbf{122} & \phm            & \phm            & \rlm y          & \phm            & \phm            & \phm            & \phm x_1        & \phm x_2        & \phm            \\
		\s \mathbf{222} & \phm            & \phm            & \phm            & \rlm y          & \phm            & \phm            & \rlm y^{b_{11}} & \rlm x_1        & \phm x_2        \\
	\end{block}
\end{blockarray}
&
\\[-1em]
\hspace{8em}
R^{7} & \filleftmap & R^{9} \leftarrow \cdots
\end{array}
$
\caption{Specialized resolution when $m = 4$ and $a_1 + a_2 = a_3$.}
\label{fig:m4ciface}
\end{figure}

\begin{remark}\label{r:citate}
With enough patience and a healthy disrespect for checking signs, the resolution given in  $\mathcal{F}_\bullet'$ can be seen as the Tate complex~\cite{tateres} resolving the field over a complete intersection. Indeed, pairs of ones and twos in the indexing multisets correspond to powers of $y_1$ and $y_2$ in the divided power algebra when the Tate complex is cast in the language of higher homotopies (see~\cite{minresoverci} and the references therein). Nevertheless, we include it here both for completeness and to indicate that such resolutions are possible to build for numerical semigroup rings without appealing to the higher homological toolkit.  
\end{remark}

\begin{thm}\label{t:m4ciface}
Under the above definitions, $\mathcal F_\bullet'$ is a resolution of $\kk$ over $\kk[S]$ that is minimal if and only if $b_{11} > 0$.  
\end{thm}

\begin{proof}
To verify $\partial^2 e_\cc = 0$, there are 4 cases to consider, based on whether $c_0 = 0$ and whether $2 \mid c_2$.  We only show the case where $c_0 = 1$ and $2 \nmid c_2$ here, as the other cases are similar.  In this case, we see
\begin{align*}
\partial^2 e_\cc
&= \partial(x_2 e_{\cc-2}' - x_1 e_{\cc-1}' + (-1)^{d-1}ye_{\cc-0}')
\\
&= x_2 (x_2 e_{\cc-22}' + x_1 e_{\cc-12}' - y^{b_{11}}e_{\cc-11}' + (-1)^{d-2}ye_{\cc-02}') \\
& \qquad - x_1 (x_2 e_{\cc-12}' - x_1 e_{\cc-11}' + (-1)^{d-1}ye_{\cc-01}') \\
& \qquad + (-1)^{d-1}y (x_2 e_{\cc-02}' - x_1 e_{\cc-01}' + (-1)^{d-2}y^{b_{22}}e_{\cc-22}')
\\
&= (x_2^2 - y^{b_{22}+1})e_{\cc-22}' + (x_1^2 - x_2 y^{b_{11}})e_{\cc-11}'
= 0.
\end{align*}

We now prove exactness.  
Fix $f \in \ker\partial_d$.  If $a$ has a nonzero term $c t^n e_\cc'$ of class~3, then $a_3 = a_1 + a_2$ implies $n - a_2 \in S$, so replacing $f$ with $f - ct^{n-a_2}\partial e_{\cc+2}'$ results in one fewer class~3 term.  Repeating this as necessary, we may assume $f$ has no class~3 terms.  Next, for each class~2 term of $f$, we may do an analogous replacement, and subsequently assume $f$ only has terms of class~0 and~1.  

Now, suppose $c t^n e_\cc'$ is a nonzero term in $f$ with class~1.  If $c_2 > 0$, then $\partial(c t^n e_\cc')$ has a class~3 term that cannot be cancelled by the image of any other terms in $f$.  As~such, we must have $c_2 = 0$, and replacing $f$ with $f - c t^{n-a_1} \partial e_{\cc+1}'$ results in one fewer term with nonzero class.  With this, we may assume every term in $f$ is class~0.  However, each term in $f$ then has one or two terms in its image under $\partial$ with nonzero class, neither of which can cancel with any other term in $\partial f$.  We conclude $f = 0$.  

For the minimality claim, $a_2 > m$ implies $b_{22} > 0$, so every term in the definition of~$\partial$ has a positive degree coefficient precisely when $b_{11} > 0$.  This completes the proof.  
\end{proof}

\subsection{A facet of $\mathcal C_4$ containing non-CI numerical semigroups}

For this subsection, suppose $a_2 = 2a_1$ and thus $x_2 = x_1^2$, so that $S = \<4, a_1, a_3\>$ lies in the upper-left facet of $C_4$ in Figure~\ref{fig:m3m4cones}.  
Resuming notation from Definition~\ref{d:infiniteaperyres}, let
$$
F_d = \<e_\ww : \ww = (w_1, w_2, \ldots, w_d) \in \ZZ_4^{d}\>,
\qquad \text{where} \qquad
\deg(e_\ww) = \sum_{i = 1}^d \deg(x_{w_i}).
$$
Let $W_d \subseteq \ZZ_m^d$ denote the set of $\ww \in \ZZ_m^d$ such that $w_1 \in \{0,1,3\}$ and for each $i \ge 2$ we have $w_i \in \{2,3\}$ if $w_{i-1} = 1$ and $w_i \in \{1,3\}$ otherwise.  
Let 
$$
F_d' = \<e_\ww : \ww \in W_d\>  \subseteq F_d$$
so that
$$\rank F_d' = 3 \cdot 2^{d-1}
\qquad \text{for} \qquad 
d \ge 1.$$

We show in Theorem~\ref{t:m4nciface} that row and column operations may be performed to the infinite Ap\'ery resolution to obtain a resolution whose first few matrices are depicted in Figure~\ref{fig:m4nciface}.  

\begin{example}\label{e:m4nciface}
The idea of the proof of Theorem~\ref{t:m4nciface} is to define maps $p_d:F_d \to F_d'$ that express each $e_\ww \in F_d \setminus F_d'$ in terms of generators of $F_d'$ in a manner that is consistent with $\partial$.  
For instance, it is natural to define
$$
p_1 e_2 = x_1 e_1
\qquad \text{since} \qquad
\partial e_2 = \partial x_1 e_1 = x_1^2,
$$
and $p_1(e_i) = e_i$ if $i \ne 2$.  
Composing $p_1$ and $\partial$ yields a map $F_2 \to F_1'$ that sends, e.g.,
\begin{align*}
e_{11} 
&\mapsto p_1 \partial e_{11} = p_1(x_1 e_1 - y^{b_{11}} e_2) = x_1 e_1 - x_1 e_1 = 0,
\\
e_{12} 
&\mapsto p_1 \partial e_{12} = p_1(x_1 e_2 - y^{b_{12}} e_3) = x_1^2 e_1 - y^{b_{12}} e_3,
\\
e_{21} 
&\mapsto p_1 \partial e_{21} = p_1(x_2 e_1 - y^{b_{21}} e_3) = x_1^2 e_1 - y^{b_{12}} e_3.
\end{align*}
Restricting $p_1\partial$ to $F_2'$ yields the second matrix in Figure~\ref{fig:m4nciface}; the second line above appears therein since $e_{12} \in F_2'$.  
From there, we may define $p_2 e_{11} = 0$ and $p_2 e_{21} = e_{12}$ based on the above, and after subsequently defining
$$
p_2 e_{23} = 0,
\qquad
p_2 e_{32} = x_1 (e_{31} - e_{13}),
\qquad
p_2 e_{22} = y^{b_{12}}e_{13},
\qquad \text{and} \qquad
p_2 e_{02} = x_1 e_{01}
$$
via analogous reasoning, restricting $p_2\partial$ to $F_3'$ yields the third matrix in Figure~\ref{fig:m4nciface}.  
\end{example}

In what follows, we say a term $c t^n e_\ww \in F_d'$ is \emph{reduced} if (i) it has class 0, or (ii) it has class 1 and $w_d = 1$.  We say an element of $F_d'$ is \emph{reduced} if all of its terms are reduced.  

\begin{lemma}\label{l:m4nciface}
Suppose there exists a map $\partial':F_d' \to F_{d-1}'$ such that:\ 
\begin{enumerate}[(i)]
\item 
the term $x_{w_d} e_{\widehat \ww}$ appears in $\partial' e_\ww$ for each $\ww \in W_d$; and 

\item 
for each $\ww \in W_d$, each term of $\partial' e_\ww$ not involving $e_{\widehat \ww}$ is reduced.  

\end{enumerate}
Then each $f \in F_{d-1}'$ is equivalent modulo $\partial' F_d'$ to a reduced element.  
\end{lemma}

\begin{proof}
For each class~3 term $c t^n e_\ww$ in $f$, we have $n - a_3 \in S$, so replacing $f$ with $f - c t^{n-a_3} \partial' e_{\ww,3}$ results in one fewer class 3 term.  Repeating this process for each class~3 term, we may assume $f$ has no terms with class~3.  Analogously, if a term $c t^n e_\ww$ in $f$ has class~2, then either $w_{d-1} = 1$, in which case $f - c t^{n-a_2} \partial' e_{\ww,2}$ has one fewer class~2 term, or $w_{d-1} \in \{2,3\}$, in which case $a_2 = 2a_1$ implies $n - a_1 \in S$ and thus $f - c t^{n-a_1} \partial' e_{\ww,1}$ has one fewer class~2 term.  In the end, we may assume $a$ also has no terms of class~2.  

Now, suppose $c t^n e_\ww$ is a class~1 term in $f$.  If $w_{d-1} \ne 1$, then $f - c t^{n-a_1} \partial' e_{\ww,1}$ has no $e_\ww$ term, and each new class~1 term is reduced by~(ii), so there is one less non-reduced term.  
As~such, after finitely many such replacements, we obtain an element equivalent to $f$ modulo $\partial' F_d'$ in which every term is reduced.  
\end{proof}

\begin{thm}\label{t:m4nciface}
Under the above definitions, 
performing appropriate row and column operations to the infinite Ap\'ery resolution yields a resolution 
\begin{equation}\label{eq:m4ncifaceres}
\mathcal F_\bullet' : 0 \longleftarrow R \longleftarrow F_1' \longleftarrow F_2' \longleftarrow \cdots
\end{equation}
for $\kk$ over $\kk[S]$.  Moreover, $\mathcal F_\bullet'$ is minimal if and only if $b_{12} > 0$.  
\end{thm}

\begin{proof}
We will construct a commutative diagram 
\[
\begin{tikzcd}
0 & R \arrow[l] \arrow[d, "p_0"] & F_1 \arrow[l, swap, "\partial_1"] \arrow[d, "p_1"] & F_2 \arrow[l, swap, "\partial_2"] \arrow[d, "p_2"] & \cdots \arrow[l] \\
0 & R \arrow[l] & F_1' \arrow[l, swap, "\partial_1'"] & F_2' \arrow[l, swap, "\partial_2'"] & \cdots \arrow[l]
\end{tikzcd}
\]
in which the bottom row is the desired resolution.  
First, define $p_0$ as the identity map on $R$, and define $\partial' e_\ww = \partial e_\ww$ for each $\ww \in W_1$.  Since $x_2 = x_1^2$, we have $\partial e_2 = \partial x_1 e_1$, so~$\partial'$ is indeed surjective onto the maximal monomial ideal of $R$.  

Proceeding by induction on $d \ge 2$, assume the maps $\partial_i'$ and $p_{i-1}$ are defined for $i < d$, and that the following hold:
\begin{enumerate}[(i)]
\item 
$\partial_{d-1}' F_{d-1}' = \ker \partial_{d-2}'$ (i.e., the bottom row is exact);

\item 
$p_{d-2}\partial_{d-1} e_\ww = \partial_{d-1}' e_\ww$ for each $\ww \in W_{d-1}$;

\item 
the term $x_{w_{d-1}} e_{\widehat \ww}$ appears in $\partial_{d-1}' e_\ww$ for each $\ww \in W_{d-1}$; and 

\item 
for each $\ww \in W_{d-1}$, each term of $\partial_{d-1}' e_\ww$ not involving $e_{\widehat \ww}$ is reduced.  
\end{enumerate}
Clearly (i)-(iv) hold in the case $d = 2$.  We now construct the maps $p_{d-1}$ and $\partial_d'$.  

First, fix $e_\ww \in F_{d-1}$.  If $d = 2$, then define 
\[
p_1 e_i = \begin{cases}
x_1 e_1 & \text{if $i = 2$;} \\
e_i & \text{otherwise.}
\end{cases}
\]
If, on the other hand, $d > 2$, then since $\partial_{d-2} \partial_{d-1} e_\ww = 0$, we have $p_{d-3} \partial_{d-2} \partial_{d-1} e_\ww = 0$.  By commutativity of the diagram, $\partial_{d-2}' p_{d-2} \partial_{d-1} e_\ww = 0$, and thus $p_{d-2} \partial_{d-1} e_\ww \in \ker \partial_{d-2}'$.  By~(i), we then have $p_{d-2} \partial_{d-1} e_\ww \in \partial_{d-1}' F_{d-1}'$.  
As such, choosing $p_{d-1} e_\ww = a_\ww$ for any element $a_\ww \in F_{d-1}'$ with $\partial_{d-1}' a_\ww = p_{d-2}\partial_{d-1} e_\ww$ ensures the diagram commutes.  We do so, subject to the following stipulations:\
\begin{itemize}
\item 
if $\ww \in W_{d-1}$, then choose $a_\ww = e_\ww$, which we may do by~(ii); and

\item 
if $\ww \notin W_{d-1}$, then $a_\ww$ is reduced, which we may do by (iii), (iv), and Lemma~\ref{l:m4nciface}.  

\end{itemize}
Having now defined $p_{d-1}$, for each $\ww \in W_d$ define
\[
\partial_d' e_{\ww} \mapsto 
x_{w_d} p_{d-1} e_{\widehat \ww} + \sum_{i=1}^{d-1} (-1)^{d-i} y^{b_{w_iw_{i+1}}} p_{d-1} e_{\tau_i \ww}.
\]
We record the following observations for each $\ww \in W_d$:
\begin{itemize}
\item 
by construction, $\partial_d' e_\ww = p_{d-1} \partial_d e_\ww$;

\item 
since $\widehat \ww \in W_{d-1}$, we have $p_{d-1} e_{\widehat \ww} = e_{\widehat \ww}$, so the term $x_{w_d} e_{\widehat \ww}$ appears in $\partial_d' e_\ww$; and

\item 
by the stipulations on $a_\ww$, each term of $\partial_d' e_\ww$ not involving $e_{\widehat \ww}$ is reduced.  

\end{itemize}
In particular, under the above definitions of $p_{d-1}$ and $\partial_d'$, the only thing left to complete the induction proof is to verify that $\partial_d' F_d' = \ker \partial_{d-1}'$.  

The first centered equation in the proof of Theorem~\ref{t:medresolution} verifies $\partial_d'\partial_{d-1}' = 0$, so in particular $\partial_d' F_d' \subseteq \ker \partial_{d-1}'$.  To~prove the reverse containment, fix~$f \in \ker \partial_{d-1}'$.  By~Lemma~\ref{l:m4nciface}, it suffices to assume $f$ has no terms of class 2 or 3, and that any class~1 term $k t^n e_\ww$ has $w_{d-1} = 1$.  We will prove that this forces $f = 0$.  Indeed, examining $\partial(k t^n e_\ww)$, the term involving $e_{\widehat\ww}$ has class 2, but since $w_{d-2} \ne 1$, no other term in $f$ has a class~2 term involving $e_{\widehat\ww}$ in its image under $\partial_{d-1}$.  This proves every term in $f$ has class~0.  
Lastly, examining the image under $\partial_{d-1}$ of each term $k t^n e_\ww$ in $f$, the term involving $e_{\widehat\ww}$ has nonzero class distinct from that appearing in the image under $\partial_{d-1}$ of any other term in $f$.  We conclude $f = 0$, thus proving exactness.  

All that remains in the minimality claim.  Since $a_2 = 2a_1$ and $a_1 \not\equiv a_3 \bmod 4$, we cannot have $a_2 = 2a_3$, so $b_{33} > 0$.  As such, every term in the image of $\partial$ has coefficient in $\kk[S]$ of positive degree unless $b_{12} = 0$.  
\end{proof}

\begin{figure}[t]
$
\begin{array}{c@{\:}c@{\:}c@{\:}c@{\:}c}
&
\begin{blockarray}{rccc}
	\\ \\
	&
	\s \mathbf{0} &
	\s \mathbf{1} &
	\s \mathbf{3} \\
	\begin{block}{r@{\,\,}[*{3}{@{\,\,}l}]}
		\s \mathbf{\emptyset} & 
		y &
		x_1 &
		x_3 \\
	\end{block}
\end{blockarray}
&&
\begin{blockarray}{@{}r@{\,\,\,}*{6}{@{}c}}
	&
	\s \mathbf{01} &
	\s \mathbf{03} &
	\s \mathbf{12} &
	\s \mathbf{\ \ 13\ \ \ \ } &
	\s \mathbf{31} &
	\s \mathbf{33} \\
	\begin{block}{@{}l@{\,\,\,}[*{6}{@{}l}]}
		\s \mathbf{0} & \phm x_1           & \phm x_3           & \phm               & \rlm y^{b_{13}}    & \rlm y^{b_{13}}    & \phm               \\
		\s \mathbf{1} & \rlm y             & \phm               & \phm x_1^2         & \phm x_3           & \phm               & \rlm x_1y^{b_{33}} \\
		\s \mathbf{3} & \phm               & \rlm y             & \rlm y^{b_{12}}    & \phm               & \phm x_1           & \phm x_3           \\
	\end{block}
\end{blockarray}
&
\\[-1em]
0 \leftarrow R & \filleftmap & R^3 & \filleftmap &
\hspace{15em}
\end{array}
$

$
\begin{array}{c@{\:}c@{\:}c@{\:}c@{\:}c@{}c}
&
\begin{blockarray}{@{}r@{\,\,\,}*{12}{@{}c}}
	&
	\s \mathbf{012} &
	\s \mathbf{013} &
	\s \mathbf{031} &
	\s \mathbf{033} &
	\s \mathbf{121} &
	\s \mathbf{123} &
	\s \mathbf{131} &
	\s \mathbf{133} &
	\s \mathbf{312} &
	\s \mathbf{313} &
	\s \mathbf{331} &
	\s \mathbf{333} \\
	\begin{block}{@{}l@{\,\,\,}[*{12}{@{}l}]}
		\s \mathbf{01} & \phm x_1^2         & \phm x_3           & \phm               & \rlm y^{b_{33}}    & \phm               & \phm               & \phm y^{b_{13}}    & \phm               & \phm x_1y^{b_{13}} & \phm               & \phm               & \rlm y^{b_{13}+b_{33}}  \\
		\s \mathbf{03} & \rlm y^{b_{12}}    & \phm               & \phm x_1           & \phm x_3           & \phm               & \phm               & \phm               & \phm y^{b_{13}}    & \phm               & \phm y^{b_{13}}    & \phm               & \phm                \\
		\s \mathbf{12} & \phm y             & \phm               & \phm               & \phm               & \phm x_1           & \phm x_3           & \phm               & \rlm y^{b_{33}}    & \phm               & \phm               & \phm y^{b_{33}}    & \phm                \\
		\s \mathbf{13} & \phm               & \phm y             & \phm               & \phm               & \rlm y^{b_{12}}    & \phm               & \phm x_1           & \phm x_3           & \phm               & \phm               & \phm               & \phm                \\
		\s \mathbf{31} & \phm               & \phm               & \phm y             & \phm               & \phm y^{b_{12}}    & \phm               & \phm               & \phm               & \phm x_1^2         & \phm x_3           & \phm               & \rlm x_1y^{b_{33}}  \\
		\s \mathbf{33} & \phm               & \phm               & \phm               & \phm y             & \phm               & \phm y^{b_{12}}    & \phm               & \phm               & \rlm y^{b_{12}}    & \phm               & \phm x_1           & \phm x_3            \\
	\end{block}
\end{blockarray}
&
\\[-1em]
\hspace{2em}
R^{6} & \filleftmap & R^{12} \leftarrow \cdots
\end{array}
$
\caption{Specialized resolution when $m = 4$ and $a_2 = 2a_1$.}
\label{fig:m4nciface}
\end{figure}

\begin{remark}\label{r:m4comparison}
There is a significant difference in the proofs of Theorems~\ref{t:m4ciface} and~\ref{t:m4nciface}.  For one, the latter is non-constructive; however, there is also a combintorial distinction.  In Theorem~\ref{t:m4nciface}, the map $\partial$ from Definition~\ref{d:infiniteaperyres} is still used, and in order to ensure minimality, care is taken to avoid the subsequence ``11'' in any summands of $F_d'$ so that $y^{b_{11}} = 1$ does not make an appearance.  In Theorem~\ref{t:m4ciface}, on the other hand, the subsequence ``12'' (which would yield a coefficient $y^{b_{12}} = 1$ in the map $\partial$), does appear in some chosen summands; minimality is instead ensured by avoiding the subsequence~``3'' and introducing an extra $x_i$ in some columns.  
\end{remark}

\begin{remark}\label{r:m4golod}
As a consequence of Theorems~\ref{t:m4ciface} and~\ref{t:m4nciface}, and the results of~\cite{kunzfiniteres}, one can see that if $m = 4$, then $R$ is either Golod or complete intersection.  Indeed, if $a_2 = 2a_1$ as in the current subsection, or if $a_2 = 2a_3$, then 
$$
P_\kk^R(z) = \frac{1 + z}{1 - 2z}
\qquad \text{and} \qquad
P_I^Q(z) = 3 + 2z
$$
by Theorem~\ref{t:m4nciface} and \cite[Example~4.2]{kunzfiniteres}, respectively, so one can readily verify that $R$ is Golod.  Aside from the interior, within which $S$ is MED and thus $R$ is Golod by Theorem~\ref{t:medresolution}, any numerical semgroup in a remaining face of $\mathcal C_4$ is complete intersection by Theorem~\ref{t:m4ciface} and the discussion at the start of Section~\ref{sec:m4spec}.  
\end{remark}

\section{The Koszul property and resolving over the associated graded}
\label{sec:associatedgraded}

When investigating the nature of infinite resolutions, one must consider the notion of Koszulness.

\begin{defn}
	A standard $\NN$-graded $\kk$-algebra $A$ is \textit{Koszul} when $\beta_{i,j}^A(\kk)=0$ whenever $i\neq j$, that is, whenever the residue field $\kk$ has linear resolution as an $A$-module. 
\end{defn}

Various criteria (though no complete characterization) exist for detecting Koszulness.  For example, if $I$ is the defining ideal of $A$, $I$ must be quadratic in order for $A$ to be Koszul. On the other hand, if $I$ has a quadratic Gr\"obner basis, then $A$ must be Koszul.  The converses for both of these statements are false, however.  More details on Koszul algebras can be found in
\cite{concakoszulsyzygies,concakoszulregularity,frobergkoszulalgebras,quadraticalgebrasbook} and the references therein.  

Upon observation, the infinite Ap\'ery resolution given in Section~\ref{sec:infiniteaperyres} is nearly linear in some sense, as the $x$ variables only appear in the differentials with exponent $1$. However, the resolution is not linear in the strictest sense, as the twists on the modules jump by more than $1$ at successive steps of the resolution due to the nonstandard grading on $R$. Indeed, the traditional notion of Koszul assumes a standard $\NN$-grading, so some generalization is necessary. We follow \cite{shellmonoid2,herzogstamatekoszulnumerical,rs10} in making such a generalization. 

\begin{defn}
The \textit{associated graded ring (with respect to $\mm$)}, written $\grmR$, is the $\NN$-graded algebra
\[\grmR = \bigoplus\limits_{i \geq 0} \mm^i/\mm^{i+1}.\]
For a non-zero polynomial $f$ in $R$, use $f^*$ to denote the homogeneous component of $f$ of lowest degree, called the \textit{initial form} of $f$, and use $I^* = \ideal{f^* \mid f \in I}$ for the ideal generated by the lowest degree components of elements in $I$. A set of polynomials $\{f_1, \ldots, f_r\}$ in $I$ is a \textit{standard basis} for $I$ if $I^* = \ideal{f_1^*, \ldots, f_r^*}$. 
\end{defn}

It can be shown that a standard basis is also a generating set for $I$. Furthermore, $I^*$ defines $\grmR$, that is, $\grmR = Q/I^*$. Details for these facts, as well as algorithms to compute $I^*$, can be checked in \cite[Sections~5.1~and~15.10.3]{Eis95}.

The associated graded ring is used to extend the notion of Koszulness to non-standard graded rings. 

\begin{defn}
An $\NN$-graded ring $R$ is \textit{Koszul} if $\grmR$ is Koszul. Furthermore, we call a semigroup $S$ Koszul if $\grm(\kk[S])$ is Koszul. 
\end{defn}

The above notion subsumes the usual definition of Koszul, since if $R$ is already standard $\NN$-graded, then $R \cong \grmR$. 


\begin{example}\label{ex:grmR}
Consider the numerical semigroups
\[
S = \<5,6,19\>,
\qquad
S' = \<5,21,69\>,
\qquad \text{and} \qquad
S'' = \<5, 31, 99\>,
\]
each of multiplicity $5$.  
The Ap\'ery sets $\{5, a_1, a_2, a_3, a_4\}$ of $S$, $S'$, and $S''$ each have $a_1$ and $a_4$ as minimal generators, and $a_2 = 2a_1$ and $a_3 = 3a_1$ as the only expressions for $a_2$ and $a_3$, so they all lie in the same face $F \subseteq C_5$.  
This shared face, depicted in Figure~\ref{fig:grmR}, implies that the generators for their defining toric ideals
\begin{align*}
I_S &= \ideal{x_1^4 - x_4y, x_1x_4 - y^5, x_4^2 - x_1^3y^4},
\\
I_{S'} &= \ideal{x_1^4 - x_4y^3, x_1x_4 - y^{18}, x_4^2 - x_1^3y^{15}},
\\ 
I_{S''} &= \ideal{x_1^4 - x_4 y^5, x_1x_4 - y^{26}, x_4^2 - x_1^3y^{21}}
\end{align*}
vary only in their $y$-exponents (the light, red dots in the poset for $F$ indicate the graded degree of these generators; see~\cite[Section~5]{kunzfaces3} for more detail).  However, their associated graded algebras, defined by 
\[
I_S^* = \ideal{x_1^5, yx_4, x_1x_4, x_4^2},
\qquad
I_{S'}^* = \ideal{x_1^4-x_4y^3, x_1x_4,x_4^2},
\qquad \text{and} \qquad
I_{S''}^* = \ideal{x_1^4, x_1x_4, x_4^2},
\]
vary significantly. For example, $I_S^*$ and $I^*_{S''}$ are purely monomial, while $I^*_{S'}$ is not.  

Looking more generally within the face containing $S$, $S'$, and $S''$, any numerical semigroup on the orange line $4a_1 = a_4 + 15$ splitting the face in Figure~\ref{fig:grmR} has assocated graded $I_{S'}^*$, any numerical semigroup in the shaded region below the orange line has associated graded $I_S^*$ or $\ideal{x_1^5, y^2x_4, x_1x_4, x_4^2}$, and any numerical semigroup in the unshaded region has associated graded $I_{S''}^*$.  
\end{example}

\begin{figure}[t]
\begin{center}
\includegraphics[height=1.8in]{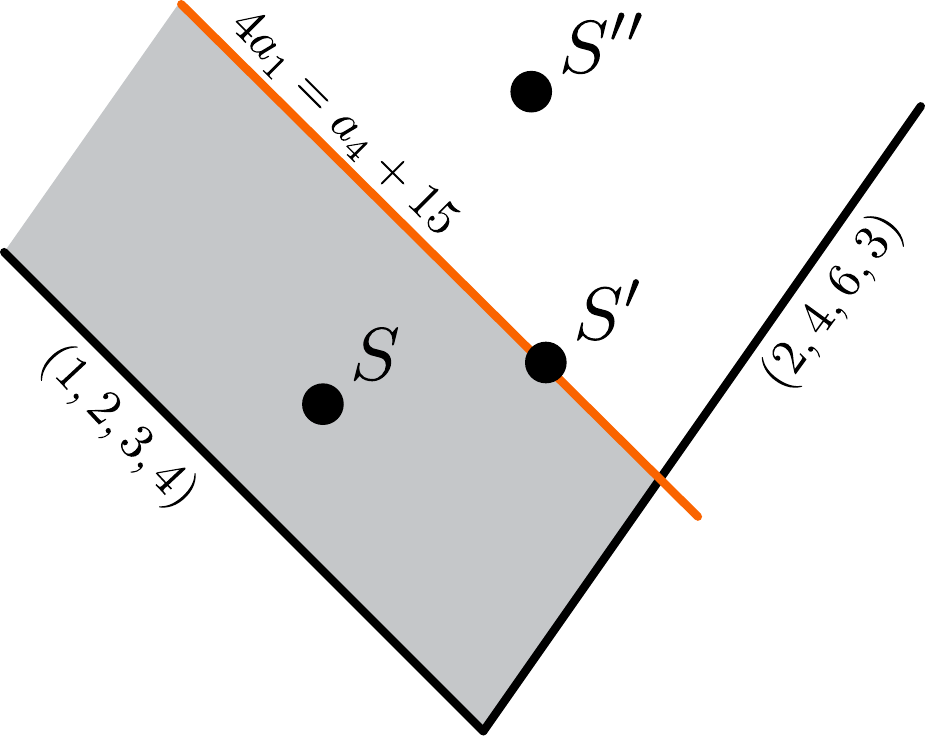}
\hspace{3em}
\includegraphics[height=1.8in]{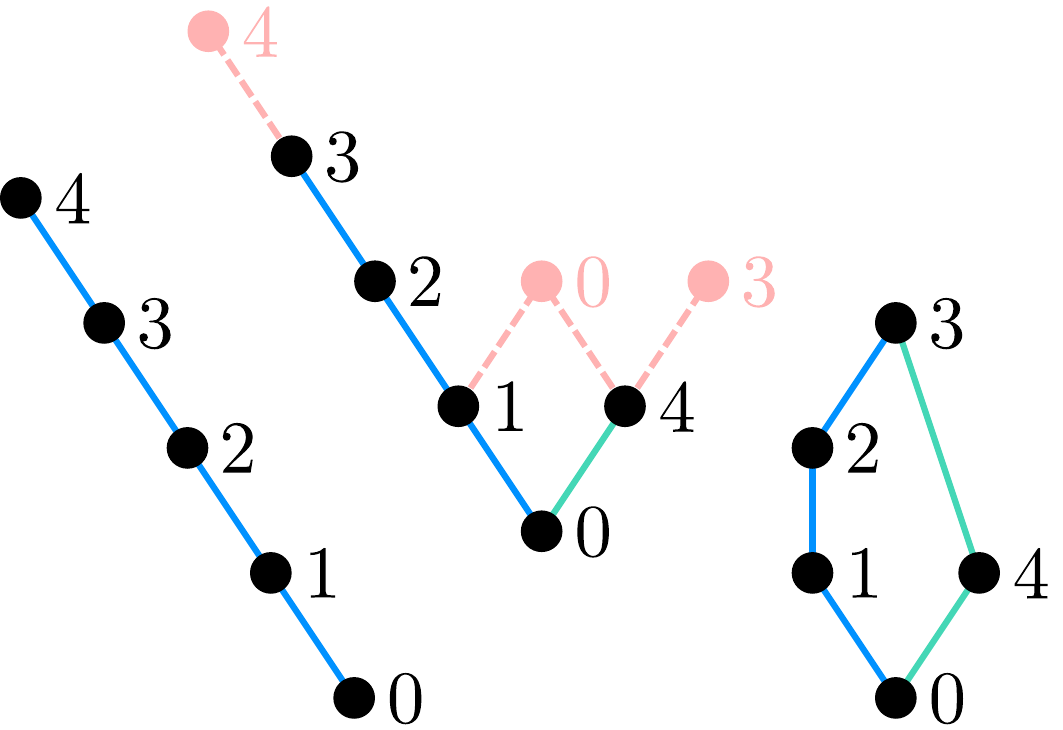}
\end{center}
\caption{The face of $C_5$ containing the semigroups in Example~\ref{ex:grmR} (left), and the Kunz posets of the face and its two rays (right).}
\label{fig:grmR}
\end{figure}

This variance can lead to some disappointment.

\begin{prop}
The Kunz cone does not distinguish between the following for semigroups coming from the same face: 
\begin{itemize}
	\item the finite Betti numbers $\beta_i^Q(\grmR)$;
	\item the Poincar\'e series of $\grmR$; and
	\item the Koszulness of $\grmR$, and therefore the Koszulness of $R$.
\end{itemize}
\end{prop} 
\begin{proof}
The semigroups $S$ and $S''$ from Example~\ref{ex:grmR} are already enough to see the first two items. Namely, $\beta_1^Q(Q/I^*_S) = 4$ and $\beta_1^Q(Q/I^*_{S''}) = 3$. Furthermore, if one computes the first few steps of the resolutions of the field, one finds that $\beta_2^{\grm(\kk[S])}(\kk) = 7$, but $\beta_2^{\grm(\kk[S''])}(\kk)=6$.

The following example gives a pair $T$ and $T'$ where $\kk[T]$ is Koszul but $\kk[T']$ is not even quadratic, but $T$ and $T'$ come from the same face of $C_m$.  Let $T = \ideal{8,9,10,12,23}$ and $T' = \ideal{8,81,90,108,207}$.  Note that $T'$ is obtained from $T$ by multiplying each generator besides $8$ by $9$, which preserves the relations within the Ap\'ery set since $9 \equiv 1 \bmod 8$, and so $T$ and $T'$ are in the same face of the Kunz cone $C_8$.  The associated graded rings in this case are quotients of $\kk[y,x_1,x_2,x_4,x_7]$ with the standard grading. 

One can check with Macaulay2~\cite{M2} that 
\[I_T^* = \ideal{x_1^2 - yx_2, x_2^2 - yx_4, x_4^2, yx_7, x_1x_7, x_2x_7, x_4x_7, x_7^2},\] 
and the generators listed form a quadratic Gr\"obner basis, so $T$ is Koszul. However, 
\[I_{T'}^* = \ideal{x_1^2,x_2^2, x_4^2, x_1x_2x_4, x_1x_7, x_2x_7, x_4x_7, x_7^2},\] 
which is not even quadratic due to the presence of $x_1x_2x_4$, and so $T'$ is not Koszul. 
\end{proof}

\begin{remark}\label{r:medkoszul}
It is worth pointing out that \cite{herzogstamatekoszulnumerical} proves that MED semigroups \textit{are} Koszul, so such pathologies can only appear on the boundary of $C_m$. In this regard, numerical semigroup rings in fixed multiplicity are ``generically'' Koszul. 
\end{remark}

\section{Open Problems}
\label{sec:futurework}

In~\cite{shellmonoid2,shellmonoid}, $\beta_i^R(\kk)$ is expressed in terms of homology of the divisibility poset of~$S$.  Corollary~\ref{c:specialization} implies these Betti numbers are in fact determined by a finite subposet, namely the Kunz poset of $S$.  This begs the following natural question, especially in light of similar results in the finite case~\cite{squarefreedivisorcomplex}.  

\begin{prob}\label{pb:shellakunzposet}
Locate a formula for $\beta_i^R(\kk)$ in terms of the Kunz poset of $S$.  
\end{prob}

By~\cite{gentoricratpoincare}, so-called ``generic'' toric rings are Golod, so Golodness is a common property among toric rings.  Although numerical semigroup rings are rarely generic in the sense of~\cite{gentoricratpoincare}, Corollary~\ref{c:medgolod} implies Golodness is still a common property within the family of numerical semigroup rings since most numerical semigroups are MED.  
As the faces of $C_m$ provide a natural means for classifying numerical semigroups, a more refined question about the prevelence of Golodness now has a combinatorial approach.  

\begin{prob}\label{pb:golodfaces}
Obtain a combinatorial (i.e., Kunz poset-theoretic) characterization of when $R$ is Golod.  How common are the faces of $C_m$ containing numerical semigroups for which $R$ is Golod?  
\end{prob}

Another property of interest is the rationality of the Poincar\'e series $P_\kk^R(z)$, which by Proposition~\ref{p:uniformpoincare} is uniform across any given face of $C_m$.  The smallest known example of a numerical semigroup $S$ for which $P_\kk^R(z)$ is irrational, given in~\cite{irrpoincare}, has $m = 18$.  An answer to the following question would yield a computational approach to determining whether smaller examples exist (see~\cite{wilfmultiplicity} for detail on face lattice computations of $C_m$ with $m \le 18$).  

\begin{prob}\label{pb:irrationalpoincarefaces}
Obtain a combinatorial (i.e., Kunz poset-theoretic) characterization of when $P_\kk^R(z)$ is rational.  
\end{prob}

By Theorem~\ref{t:medresolution}, each summand in the minimal resolution of $\kk$ over $R$ corresponds to a word $\ww \in \ZZ_m^d$.  In the nomenclature of theoretical computer science, one may define a \emph{language} $\mathcal L_S$ consisting of all such words $\ww$ (note that $\mathcal L_S$ is dependent only on the Kunz poset of $S$ by Corollary~\ref{c:specialization}).  
By the results in Section~\ref{sec:m4spec}, for any numerical semigroup with $m = 4$, $\mathcal L_S$ is a regular language, which implies $P_\kk^R(z)$ is rational (see, e.g., the discussion and references in~\cite{combofromratgen}).  
The following, although a loftier goal than the above problems, is thus one possible approach to Problem~\ref{pb:irrationalpoincarefaces}.  

\begin{prob}\label{pb:language}
Classify the language $\mathcal L_S$ in terms of the Kunz poset of $S$.  
\end{prob}

Cellular resolutions are a popular tool from combinatorial commutative algebra used to resolve monomial ideals over the polynomial ring~\cite{BPS98,BS98}.  For example, the Koszul resolution of the field over the polynomial ring can be seen as a cellular resolution supported on the simplex.  Given the combinatorial nature of numerical semigroup rings, after some preliminary computations, we make the following conjecture.

\begin{conj}\label{conj:cellularaperyres}
One can ascribe a cellular structure to the infinite Ap\'ery resolution.  
\end{conj}

Given the consistency of the resolution of $\kk$ across each face of $C_m$, do there exist other families of modules (e.g., quotients by certain monomial ideals) whose resolutions are also consistent across the face?  For instance, after some preliminary computations, the Betti numbers of the quotient $M = R/\<x_1, \ldots, x_{m-1}\>$, resolved over $R$, appear to be consistent across the face (note that $M$ is spanned over $\kk$ by finitely many powers of $y$, but its dimension over $\kk$ can vary for semigroups in the same face of $C_m$).  

\begin{prob}\label{prob:idealfamilies}
Locate families of ideals $I \subseteq R$ for which the resolutions over $R$ are consistent across the face of $C_m$ containing $S$.  
\end{prob}

To this end, we also pose the following, inspired by the resolutions constructed in~\cite{taylorci}.  

\begin{prob}
Does there exist a Taylor-like resolution for monomial ideals over $R$?  
\end{prob}




\section*{Acknowledgements}

The authors would like to thank David Eisenbud for his feedback and several helpful conversations, and in particular for identifying portions of \cite{kunz} that had been overlooked for many years.


\end{document}